\documentclass[11pt]{article}
\usepackage{ae} % or {zefonts}
\usepackage[T1]{fontenc}
\usepackage[ansinew]{inputenc}
\usepackage[reqno,namelimits]{amsmath}
\usepackage{amsthm}
\usepackage{amssymb}
\usepackage{graphicx}
\usepackage{subfigure}
\newtheorem{thm}{Theorem}
\newtheorem{corollary}[thm]{Corollary}
\newtheorem{lemma}[thm]{Lemma}
\newtheorem{proposition}[thm]{Proposition}
\newtheorem{definition}[thm]{Definition}
\newtheorem{remark}{Remark}
\newtheorem{example}[thm]{Example}

\numberwithin{equation}{section}

\author{N. Kadianakis \footnote{Department of Mathematics, National Technical University of Greece, Zographou Campus, 157 80 Athens Greece,  E-mail: nkad@math.ntua.gr} and F. I. Travlopanos \footnote{Sofianopoulou 11-13, K. Patissia 111 45, Athens, Greece, E-mail: ftravlo@gmail.com}}

\title{Variation of the Affine Connection in the kinematics of hypersurfaces}
\begin{document}
\maketitle
\begin{abstract}
 Affine deformations serve as basic examples in the continuum mechanics of deformable 3-dimensional bodies (referred as homogeneous deformations). They preserve parallelism and are often used as an approximation to general deformations. However, when the deformable body is a membrane, a shell or an interface modeled by a surface, the parallelism is defined by the affine connection of this surface. In this work we study the infinitesimally affine time - dependent deformations (motions) after establishing formulas for the variation of the connection, but in the more general context of hypersurfaces of a Riemannian manifold. We prove certain equivalent formulas expressing the variation of the connection in terms of geometrical quantities related to the variation of the metric, as expected, or in terms of mechanical quantities related to the kinematics of the moving continuum. The latter is achieved using an adapted version of polar decomposition theorem, frequently used in continuum mechanics to analyze the motion. Also, we apply our results to special motions like tangential and normal motions. Further, we find necessary and sufficient conditions for this variation to be zero (infinitesimal affine motions), giving insight on the form of the motions and the kind of hypersurfaces that allow such motions. Finally, we give some specific examples of mechanical interest which demonstrate motions that are infinitesimally affine but not infinitesimally isometric.
\end{abstract}
%%% ----------------------------------------------------------------------
%                              MAIN DOCUMENT
%%% ----------------------------------------------------------------------

\section{Introduction}
Affine or homogeneous motions of a 3-dimensional body moving in the 3-dimensional Euclidean space are important examples in continuum mechanics because they are simple enough, preserve parallel directions and can be used to approximate general motions. However, when the body is a membrane, a shell or an interface, modeled by a surface, and therefore having a non Euclidean structure, the parallelism is defined by the affine connection of the surface. Hence, motions preserving the affine connection and therefore parallelism, become important. In this work we study infinitesimally affine motions from a more general perspective, by considering an m-dimensional continuum modeled by a hypersurface, moving in an (m+1)-dimensional ambient space, modeled by a Riemannian manifold. Although it is expected that the variation of the affine connection must be related to the variation of the metric of the hypersurface, in this paper we prove exact formulas for this relation. We give formulas for the variation of the connection either in terms of quantities from continuum mechanics or in terms of geometrical quantities related to the variation of the metric. These formulas are then used to state conditions for a motion in order to have zero variation of the connection (infinitesimal affine motion). The variation of an affine connection has been studied before, either for specific ambient spaces, or specific motions, or from a purely geometric viewpoint (\cite{Guven1}, \cite{Guven2}, \cite{Yano2}). In this work we use the most general motion as well as a general ambient space, and our formulas are given in a frame - independent way and are related to  geometry or to continuum mechanics. Further, our results comply with to those in the above works for suitable ambient spaces or motions.

In section 2 we review the main facts for the geometry of a hypersurface. In section 3 we present some basic concepts of the continuum mechanics related to a version of the polar decomposition theorem for moving hypersurfaces. In section 4 we define the notion of the variation of the connection. In section 5 we derive two types of formulae for the variation of the connection. The first type expresses the variation using the kinematical quantity of stretching which is essentially the variation of the metric and it is related to the symmetric part of the velocity gradient. The second type uses the second fundamental form of the hypersurface. It is shown that a hypersurface can undergo an infinitesimally affine motion if and only if the stretching of the motion, equivalently the variation of the metric, is covariantly constant (parallel). This fact imposes a restriction for the hypersurface in order to admit a parallel symmetric tensor field (\cite{Eisenhart}, \cite{Vilms}). We also study the case of tangential motion in which the velocity of the motion is tangent to the hypersurface and the normal motion in which the velocity is along the normal to the hypersurface. We prove that in the particular case of normal motion in which the hypersurface is moving parallel to itself, the variation of the connection is explicitly given by the covariant derivative of the second fundamental form. This fact imposes a strong restriction on the kind of hypersurface admitting an infinitesimally affine and parallel to itself motion. In section 6 we use the well known case of parallel hypersurfaces in a Euclidean space, in order to give an example of explicit calculation of the variation of the connection. Also we give two examples of mechanical interest in which although the metric varies, the motion is infinitesimally affine: The first example is a model for spherical balloon expanding by blowing in air (homothetically expanding sphere) and another one of unrolling and stretching a piece of cylindrical shell. 

This work contributes to the general study of the variation of the intrinsic and extrinsic geometry of a hypersurface, (\cite{NKadianakis}, \cite{NKad_Ftravlo}) and it is an attempt towards a rational and coordinate-free description of the kinematics of continua (see also \cite{Gurtin}, \cite{Murdoch1989}, \cite{Noll} and Truesdell \cite{Truesdell}). For this purpose we use an adapted version of the polar decomposition theorem introduced in \cite{COHEN}.  
The deformation of hypersurfaces has many applications in continuum mechanics, as in the description of shells and membranes, material surfaces, interfaces ( \cite{Szwabo}, \cite{Guven1}), thin rods constrained on surfaces \cite{Heijden} and wavefronts, as well as in the theory of relativity \cite{Guven2}. Additionally, the use of differentiable manifolds and Riemann spaces in continuum mechanics (\cite{Marsden}, \cite{Epstein}, \cite{Betounes1}) gives a better understanding of the concepts of continuum mechanics as well as allows the description of more complex situations (\cite{Yavari1}, \cite{Yavari2}). Finally, the subject has also been studied in Differential Geometry (\cite{Yano2}, \cite{BenAndrews}, \cite{DoCarmo2}.

\section{The geometry of a hypersurface}
e consider an ($m+1$) - dimensional Riemannian manifold $N$ and an $m$ - dimensional, oriented, differentiable hypesurface $M$ of $N$ and its canonical embedding
$j:M\hookrightarrow{N},$ writing $j(M)=\widetilde{M}\subset N.$ We denote by $\bar{g}$, $\overline{\nabla}$, $\bar{R}$ the metric tensor, the associated Levi Civita connection and  the Riemann curvature respectively of the ambient manifold $N$, while $g$, $\nabla$ and $R$ are denoting the corresponding geometrical quantities on the hypersurface $M$.
For each $X\in M$ let $J_{X}=dj_{X}:T_{X}%
{M}\rightarrow T_{j(X)}N$ be the differential of $j$ at $X.$ The sets of vector fields on
$M$ and $N$ are denoted by ${{{{\mathcal{X}}}}}(M)$ and ${{{\mathcal{X}}}}(N)$  respectively, and ${{{\mathcal{\bar{X}}}}}(M)$ is the set of vector fields defined on $M$ with values on the tangent bundle of the ambient manifold $N$ (also called vector fields along $M$). 
If $u\in{{{{\mathcal{X}}}%
}}(M)$, then $\bar{u}=Ju\in{{{\mathcal{\bar{X}}}}}(M),$ while if $\bar{u}\in{{{\mathcal{X}}}}(N)$ 
we may define its restriction to $M$ by
$w=\bar{u}\circ j\in{{{\mathcal{\bar{X}}}}}(M).$ Similarly, for a vector
field $w\in{{{\mathcal{\bar{X}}}}}(M)$ one can define an extension of it
$\bar{w}\in{{{\mathcal{X}}}}(N)$ as a vector field $\bar{w}$ such
that $\bar{w}\circ j=w.$ Extensions are not unique but their existence is
guaranteed. In what follows we assume that the fields along $M$ are
already extended to fields on $N$. For the Lie bracket we have (\cite{Yano}, p. 88) that if $u,v\in{{{{\mathcal{X}}}}}(M)$, and $Ju,Jv$ are extended
to fields on $N$, then
\begin{equation}
J[u,v]=[Ju,Jv]. \label{JLie}%
\end{equation}
A unit normal vector field $n\in\mathcal{\bar{X}}(M)$\ of the oriented
hypersurface satisfies:
\begin{equation}
\bar{g}(n,n)=1,\ \ \ \ \ \ \bar{g}(Ju,n)=0,\ \ u\in{{{\mathcal{X}}}}(M).
\label{n}%
\end{equation}
The induced metric tensor field $g$ on $M$ (\textit{first fundamental form})
is given by: %$g=j^{\ast}\bar{g}$ i.e.%
\begin{equation}
g(u,v)=\bar{g}(Ju,Jv),\ \ \ \forall u,v\in{{{\mathcal{X}}}}(M). \label{1stff}%
\end{equation}
For each $X\in M$ we have the decomposition $T_{j(X)}N=J_{X}(T_{X}{M})\oplus N_{X},$
where $N_{X}=span\{n_{X}\}$ is the one dimensional subspace of $T_{j(X)}N$
generated by the unit normal $n_{X}$ at $X$. For any $W\in T_{j(X)}N$ the
normal projection (projection along$\ n)$ is the map
\begin{equation}
\pi_{X}:T_{j(X)}N\longrightarrow T_{j(X)}N,\ \ \ \pi_{X}(W)=W-\bar{g}(W,n)n.
\label{pi}%
\end{equation}
Since $\pi_{X}(W)\in T_{j(X)}\widetilde{M},$ it is the image under $J_{X}$ of
a vector $w\in T_{X}M$, that is, $\pi_{X}(W)=J_{X}w$. Thus we can define the projection: ${\mathcal P}_{X}:T_{j(X)}N\longrightarrow T_{X}M$, ${\mathcal P}_{X}W=w.$ We call $w$ the projection of $W$ to $T_{X}M$. 
Then it can be shown that: 
\begin{align}
J_{X}{\mathcal P}_{X}  &  =\pi_{X}:T_{j(X)}N\rightarrow T_{j(X)}N,\label{JP}\\
{\mathcal P}_{X}J_{X}  &  =I_{X}:T_{X}M\rightarrow T_{X}M,\label{PJ}\\
{\mathcal P}_{X}n_{X}  &  =0. \label{Pn}%
\end{align}
The induced metric $g$ on $M$ gives rise to the associated Levi-Civita
connection $\nabla$ on $M$ such that $\nabla g=0$ and which is defined by (omitting the point $X$), 
\begin{align*}
\nabla_{u}w= {\mathcal P}\overline{\nabla}_{Ju}Jw, \, \, \, \pi \overline{\nabla}_{Ju}Jw &= \overline{\nabla}_{Ju}Jw - \bar{g} \left( \overline{\nabla}_{Ju}Jw, n \right)n ~\text{~}\, \forall\;  u,w\in{{{{\mathcal{X}}}}}(M).
\end{align*}
The\textit{\ shape operator}
or the \textit{Weingarten map} is a symmetric linear map defined at each $X$
$\in M$ by
\begin{equation}
S_{X}:T_{X}M\rightarrow T_{X}M,\ \ \ \ \ \ S_{X}u=-{\mathcal P} \overline{\nabla}_{J_{X}u}n.
\label{ShapeOperator}%
\end{equation}
The \textit{second fundamental form} and the {\it third fundamental form} are the symmetric $(0,2)$ tensor fields on
$M$ defined respectively by:
\begin{align}
B(u,w)&=g(Su,w)= \bar{g}(n,\overline{\nabla}_{Jw}Ju). \label{2ndff} \\
III(u,w)&=g(Su,Sw)=B(Su,w). \label{3rdff}
\end{align}
The \textit{Gauss equation} relates the connections $\nabla$ of the hypersurface and
$\overline{\nabla}$ of the ambient space:
\begin{equation}
\overline{\nabla}_{Ju}Jw=J\nabla_{u}w+B(u,w)n=J\nabla_{u}w+g(Su,w)n
\label{GaussEquation}.%
\end{equation}
The Riemannian curvature for the connection $\nabla$ is defined, for any pair of
vector fields $u,v\in{{{{\mathcal{X}}}}}(M)$, as the linear map $R(u,v):{{{{\mathcal{X}%
}}}}(M)\rightarrow{{{{\mathcal{X}}}}}(M)$ such that:%
\begin{equation}
R(u,v)w=\nabla_{u}\nabla_{v}w-\nabla_{v}\nabla_{u}w-\nabla_{\lbrack u,v]}w.
\label{InducedCurvature}%
\end{equation}
Similarly $\overline{R}$ denotes the Riemann curvature tensor of the affine connection $\overline{\nabla}$ of the ambient space and it is related to the curvature tensor of the hypersurface by means of the projection, i.e.
\begin{align}\label{Gauss_equation_hypersurface}
{\mathcal P} \bar{R}(Ju,Jv)Jw  &  =\ R(u,v)w+g(Su,w)Sv-g(Sv,w)Su, \\
\bar{R}(Ju,Jv)n &= J (\nabla_{u}Sv - \nabla_{v}Su). \label{Gauss_equation_hypersurface_n}
\end{align}
Using the metric $g$ we can associate a vector field $u\in{{\mathcal{X}}}(M)$ with a 1-form $u^{\flat}$ on $M$, a 1-form $\xi$ with a vector field
$\xi^{\sharp}\in{{\mathcal{X}}}(M),$ such that for each vector field $v\in{{{\mathcal{X}}}}(M)$:
\begin{align}
u^{\flat}(v)  &  =g(u,v),  g(\xi^{\sharp},v) =\xi(v).\label{yfesidiesivectors}
\end{align}
The differential $df$ of a smooth real function $f$ has as associated vector field its gradient i.e. $\nabla f=(df)^{\sharp}$. These operations are the
usual operations of raising and lowering of indices. Extending them to tensor fields, for any linear map $T: T_{X}M \rightarrow T_{X}M$ we have $T^{\flat}(u,v) = g(Tu,v)$. In particular the $B = S^{\flat}$ holds. Further, the operation of lowering indices commutes with the coavariant differentiation, that is:
\begin{align}\label{commut_lower_cov_der}
(\nabla_{v}T)^{\flat} &= \nabla_{v}T^{\flat}, \, \, \forall v \in {\mathcal X}(M).
\end{align} 
The Codazzi equation \cite{DoCarmo}, in terms of the shape operator or the second fundamental form, takes for any $u, v \in {{\mathcal{X}}}(M),$ one of the following forms respectively:%
\begin{align}\label{codazzi S}
\begin{array}{cc}
(\nabla_{v}S)u -(\nabla_{u}S)v = {\mathcal P} \bar{R}(Ju,Jv)n, \\
(\nabla_{u}B)(v,w) - (\nabla_{v}B)(u,w) = \bar{R}^{\flat}(Ju,Jv,Jw, n)
\end{array}
\end{align}

\section{Kinematics of a hypersurface}

For the kinematics of a hypersurface we use concepts from continuum mechanics assuming that the continuum, or the material body, in question has a configuration $M$ which is a hypersurface in a Riemannian manifold $N$ (the ambient space) and study deformations of the configuration $M$ of the body. This situation may be physically motivated by assuming $M$ is a surface (modelling a membrane) moving in three dimensional Euclidean space or a material curve moving on a surface $N$. The points $X\in M$ are referred to as material points. 
\begin{definition}
We call a \textit{deformation} of a hypersurface ${M}$ in the Riemannian
manifold $N$ an embedding
\begin{align}
\phi:M\ni X\longrightarrow\ x=\phi(X)\in N
\end{align}
of $M$ in $N$. We call $\widetilde{M}=\phi(M)$ the deformed hypersurface.
\end{definition}
We denote by $\widetilde{\phi}:M\rightarrow\widetilde{M}$
the induced diffeomorphism between $M$ and $\widetilde{M}$ . If\ $j:\widetilde
{M}\rightarrow N$ is the canonical embedding of $\widetilde{M}$ in $N$, then
$\phi=j\circ\widetilde{\phi}$. Considering the differentials  $ F(X)$, $\widetilde{F}(X)$,  $J_{x}$, of $\phi$%
,$~\tilde{\phi}$ and $j$, respectively, we have that:
\begin{equation}
F(X)=J_{x}\widetilde{F}(X). \label{FFtilda}%
\end{equation}
The linear map $F(X)$ is called in continuum mechanics the \textit{deformation
gradient at }$X$.
\begin{definition}
A { \it motion} of a hypersurface $M$ in a Riemannian manifold $N$ is a
$1$-parameter family of embeddings $\phi_{t}$, $t \in {I}$, where $I$ is an open interval in ${\mathbb R}$. Equivalently the motion is given by the map
\begin{align}
	\phi:M\times{I} \rightarrow N, x=\phi(X,t)=\phi_{t}(X) \in N
\end{align}
\end{definition}
We denote by $M_{t}=\phi_{t}(M)$ the deformed hypersurface at time $t$. The
mappings $\tilde{\phi}_{t}:M\rightarrow M_{t},\ $such that $\phi_{t}%
=j_{t}\circ\widetilde{\phi_{t}}$ are diffeomorphisms for each $t$ and
$\phi(X,t)=j_{t}(\tilde{\phi}(X,t)).$ The differential of $\tilde{\phi}$ at
$X$ is denoted by $\tilde{F}(X,t):T_{X}M\rightarrow T_{x}M_{t}.$ The map
$j_{t}:M_{t}\rightarrow N$ is the canonical embedding at time $t$ of the
hypersurface $M_{t}$ having differential $J_{t}(x)=dj_{t}(x):T_{x}M_{t}\rightarrow T_{j(x)}N.$ The velocity of the
material point $X$ at time $t$ is the velocity $V(X,t)$ of the curve $\phi
_{X}:%
%TCIMACRO{\U{211d} }%
%BeginExpansion
\mathbf{R}
%EndExpansion
\rightarrow N,~~\phi_{X}(t)=\phi(X,t)$ i.e.
\begin{align*}
V(X,t)=\frac{\partial}{\partial t}\phi_{X}(t).
\end{align*}
The \textit{velocity field }of the motion is the map $V(\cdot,t):M\rightarrow
TN$, i.e $V\in\mathcal{\bar{X}}(M).$ We often write $V(X,t)=\frac{\partial
}{\partial t}\phi(X,t)$. The \textit{spatial velocity} at the point
$x=\phi(X,t)$ is the vector field $v(\cdot,t):M_{t}\rightarrow TN,$ given
by,$\ $%
\begin{equation}
v(x,t)\mathbf{=}V(\tilde{\phi}_{t}^{-1}(x),t)\ \ \ \text{i.e.}\ \ \ v(\tilde
{\phi}(X,t),t)=V(X,t). \label{vV}%
\end{equation}
We now define the \textit{velocity gradient} of the motion as the map,
\begin{equation}
G(x):T_{x}M_{t}\rightarrow T_{j(x)}N, \, \, G(x)u =\overline{\nabla}_{Ju}v=\overline{\nabla}v(Ju). \label{G defined}%
\end{equation}

 It is useful,  rather than using the motion $\phi_{t}$, to use the map $\phi_{t}(\cdot,\tau):M_{t}\rightarrow N$ representing the deformation from the present configuration $M_{t}$ to a
future one at time $\tau$ and such that:%
\begin{equation}
\phi(X,\tau)=\phi_{t}(\tilde{\phi}(X,t),\tau). \label{x tt}%
\end{equation}
Assuming that the points $X$ or $x$ are implied from the definition of the maps involved, we may omit them.
The trajectory of the point $x\in M_{t}$ is given by the mapping
\begin{equation}
\phi_{t}(x):\mathbf{R\rightarrow}N,~~\phi_{t}(x)(\tau)=\phi_{t}(x,\tau)\in
N,~~\phi_{t}(x,t)=j_{t}(x). \label{trajectory of x}%
\end{equation}
It can be shown that the spatial velocity defined by (\ref{vV}) is the vector
field associated with $\phi_{t}(.,\tau),$ i.e.%
\begin{equation}\label{spatial_velocity}
v_{x}(t)=\frac{\partial}{\partial\tau}|_{\tau=t}\phi_{t}(x,\tau).
\end{equation}

We write $F(\tau),$ $F(t),$ for the deformation gradients corresponding to the
times $\tau$ and $t$ respectively, and $F_{t}(\tau)(x_{t})$ for the space
differential of $\phi_{t}(.,\tau):M_{t}\rightarrow N$ at $x_{t}=\tilde{\phi
}(X,t)$, i.e.
\begin{align*}
F_{t}(\tau)(x_{t}) &= d\phi_{t}(x_{t},\tau):T_{x_{t}}M_{t}\rightarrow T_{x_{\tau
}}N,
\end{align*}
which we call the \textit{relative deformation gradient.} The time derivative of $F_{t}(\tau)$ is defined for each $x \in M_{t}$ along the trajectory of $x$ using the covariant derivative of the ambient space (\cite{DoCarmo}, pg. 50) by 
\begin{align}
\left( \frac{\partial}{\partial\tau}|_{\tau=t}F_{t}(\tau)\right) u &= \bar{\nabla}_{v}F_{t}(\tau)u \label{dF=G}.%
\end{align}
By interchanging space differential and time derivatives and using (\ref{spatial_velocity}), equation (\ref{dF=G}) can be written
\begin{align}\label{dF=G_1}
Gu &= \bar{\nabla}_{v}F_{t}(\tau)u.
\end{align}
The polar decomposition theorem has a long history in continuum mechanics. Let
$E=N$ be a three dimensional Euclidean space, $V$ its associated vector space and $M$ an open region in $E$, 
representing the reference configuration of a material body. If $\phi:M\rightarrow E$ is a deformation of $M$, the theorem
states that the deformation gradient $F(X):T_{X}M=V\rightarrow T_{\phi(X)}N=V$
is decomposed uniquely as $F(X)=R(X)U(X)$, where $R$ and $U$ are linear maps
of $V,$ with $U$ being symmetric and positive definite and $R$ being orthogonal. Hence
the deformation gradient $F$ is decomposed into a pure deformation followed by a
rotation. 
If the body has codimension 1, then $F(X)$ is a map between spaces having different dimensions and the last formula has been modified by C.-S. Man and H. Cohen \cite{COHEN} to read:%
\begin{align*}
F(X)=R(X)J_{X}U(X),
\end{align*}
where $R(X):V\rightarrow V$ is an orthogonal linear map, $U(X):T_{X}%
M\rightarrow T_{X}M$ is a linear symmetric and positive definite and $J_{X}=dj(X):T_{X}%
M\rightarrow V$ is the differential of the canonical embedding $j:M\rightarrow
E$ of the surface into the Euclidean ambient space. Then $U^{2}(X)=F^{*}%
(X)F(X)=\widetilde{F}^{*}(X){*}\widetilde{F}(X)$, where $F^{\star}$ is the adjoint of $F$ relative to the metrics of the hypersurface and the ambient space. Also, the unit normal vector
fields $n(X)$ on the original surface at the point $X$ and $n(x)$ on the
deformed surface at the point $x=\phi(X),$ are related by:%
\begin{equation}
n(x)=R(X)n(X). \label{Rn}%
\end{equation}
This theorem has been used in \cite{NKadianakis} to derive formulae for the
variation of geometrical quantities of a surface. 
Generalizing this to manifolds in \cite{NKad_Ftravlo} and by applying it to the relative deformation gradient $F_{t}(\tau)$ we get:
%(\tau):T_{x_{t}}M_{t}\rightarrow T_{x_{\tau}}N$, we have%
\begin{equation}
F_{t}(\tau)=R_{t}(\tau)J_{t}U_{t}(\tau), \label{PDT}%
\end{equation}
where $C_{t}(\tau)=U_{t}^{2}(\tau)=F_{t}^{*}(\tau)F_{t}(\tau):T_{x}%
M_{t}\rightarrow T_{x}M_{t}$, $U_{t}(\tau)$ being the \textit{relative right
stretch tensor }and $R_{t}(\tau):T_{x_{t}}N\rightarrow T_{x_{\tau}}N,$ with
$x_{t}=j_{t}(x)$ and $x_{\tau}=\varphi_{t}(\tau)(x)$ being the \textit{relative
rotation tensor}. Further, the unit normal fields $n(t)$ on $M_{t}$ and
$n(\tau)$ on $M_{\tau}$ are related via the rotation $R_{t}(\tau)$ by:
\begin{equation}
n(\tau)=R_{t}(\tau)n(t) \label{n(tau)_n(t)}.%
\end{equation}
For $t=\tau,\ $ we have
\begin{equation}
F_{t}(t)=J_{t},\ \ \ \ R_{t}(t)=I_{T_{x_{t}}N},\ \text{\ \ }U_{t}%
(t)=I_{T_{x}M_{t}}. \label{t=tau1}%
\end{equation}
From (\ref{x tt}) it follows that%
\begin{equation}
F(\tau)=F_{t}(\tau)\tilde{F}(t)=F_{t}(\tau)P_{t}F(t).
\label{F(tau)=F(tau t)PtF(t)}%
\end{equation}
For each $x\in M_{t}$ the \textit{stretching }of the motion is given by:%
\begin{equation}
\mathcal{D}(t)=\frac{\partial U_{t}(\tau)}{\partial\tau}|_{\tau=t}=\frac{1}%
{2}\frac{\partial C_{t}(\tau)}{\partial\tau}|_{\tau=t}:T_{x}M_{t}\rightarrow
T_{x}M_{t}. \label{D(t)}%
\end{equation}
Let $u \in {\mathcal X}(M)$, then $Ju \in \overline{{\mathcal X}}(M)$ and we define a vector field $\bar{u} \in {\mathcal X}(N)$ by setting:
\begin{align}\label{extend_vf_by_motion_defn}
\bar{u} (\phi_{t}(\tau)(x)) &= F_{t}(\tau)u(x) \equiv u_{t}(\tau)(x),
\end{align}  
then
\begin{align}\label{extend_vf_by_motion_eqn1a}
\bar{u}(\phi_{t}(t)(x))&= \bar{u}(j(x)) = Ju.
\end{align}
Since $\bar{u}$ is defined by the motion itself:
\begin{align}\label{extend_vf_by_motion_eqn2a}
[ v, \bar{u}] &= 0.
%\label{extend_vf_by_motion_eqn2b}
%Gu &= \overline{\nabla}_{Ju}v = \overline{\nabla}_{\bar{u}}v = \overline{\nabla}_{v}\bar{u}.
\end{align} 
Using (\ref{spatial_velocity}) it follows that (\cite{DoCarmo}, p.50) the covariant derivative along the trajectory of $x\in M_{t}$ is
\begin{align}\label{extend_vf_by_motion_eqn1c}
\overline{\nabla}_{v}\bar{u} &= \frac{\partial u_{t}(\tau)}{\partial \tau}|_{\tau =t}(x).
\end{align}
Therefore we have the following expressions for the velocity gradient:
\begin{align}\label{extend_vf_by_motion_eqn2b}
Gu &= \overline{\nabla}_{Ju}v = \overline{\nabla}_{\bar{u}}v = \overline{\nabla}_{v}\bar{u}.
\end{align}
Similarly the time derivative of the tensor field $R_{t}(\tau)$ is defined for each $x\in M_{t}$ along the trajectory of $x,$ by
\begin{equation}
W(t)=\frac{\partial R_{t}(\tau)}{\partial\tau}|_{\tau=t}=\overline{\nabla}%
_{v}R_{t}(\tau)|_{j(x)}:T_{j(x)}N\rightarrow T_{j(x)}N. \label{W(t)}%
\end{equation}
Some of the basic relations between the above kinematical quantities are summarized in the following lemma proved in  \cite{NKad_Ftravlo}.
\begin{lemma}
For a fixed time $t$, let $\phi_{t}(\tau)$ be a relative motion of the hypersurface
$M_{t}$ in the Riemannian manifold $N$ with velocity vector field ${v}$ which
splits into tangential and normal parts as: $v=v_{||}+v_{n}{n}=J{v}^{||}%
+v_{n}{n}$, where ${v}^{||}\in{{{{\mathcal{X}}}}}(M_{t})$. Then:
\begin{align}
G &= J {\mathcal D}+WJ, \, \, {\mathcal P}G = {\mathcal D} + {\mathcal P}WJ = \nabla v^{||} -v_{n}S \label{G=JD+WJ}\\
2\mathcal{D} &  =\nabla{v}^{||}+\nabla{v}^{||^{\star}}-2v_{n}S,\label{DandS}\\
{\pounds }_{v^{||}} g &= 2\mathcal{D}^{\flat}+2v_{n}B  = ( \nabla v^{||} + \nabla v^{||^{*}})^{\flat} %
.\label{DgB}
\end{align}
\end{lemma}
\section{The notion of variation}
In \cite{NKad_Ftravlo} we have defined a $\tau$ - dependent geometry on the instantaneous hypersurface $M_{t}$ by pulling back on $M_{t}$ the geometry of $M_{\tau}$, using the motion. Using this procedure we have defined the $\tau$ - dependent metric tensor and the $\tau$-dependent shape operator $S_{t}(\tau):$
\begin{align}
g_{t}(\tau)(u,w)  &  =\bar{g}(F_{t}(\tau)u,~F_{t}(\tau)w)\label{gtt}\\
F_{t}(\tau)S_{t}(\tau)u &  =-\overline{\nabla}_{F_{t}(\tau)u}n(\tau);
\label{tt shape2}%
\end{align}
The $\tau$ - dependent Levi - Civita connection $\nabla_{t}(\tau)$ on $M_{t}$, is defined in a similar way by using the motion $\phi_{t}(\tau)$: 
\begin{align}\label{sec2_var_con_eqn1b}
F_{t}(\tau)\nabla_{t}(\tau)_{u}w &= \overline{\nabla}_{F_{t}(\tau)u}F_{t}(\tau)w - \overline{g} \left( \overline{\nabla}_{F_{t}(\tau)u}F_{t}(\tau)w, n(\tau) \right)n(\tau),
\end{align}   
One can show, using typical arguments, that this connection is actually the unique symmetric connection compatible with the pulled-back metric $g_{t}(\tau)$. This means that the $\tau$ - dependent Levi - Civita connection is equivalently given by the $\tau$ - dependent metric tensor from the formula (\cite{DoCarmo}, p. 50) 
\begin{align}\label{sec2_var_Christof_eqn1b}
2 g_{t}(\tau)( \nabla_{t}(\tau)(u,w),z ) &= u g_{t}(\tau) ( w, z ) + w g_{t}(\tau)( u , z ) -  z g_{t}(\tau)( u, w ) \nonumber \\
&- g_{t}(\tau)( u, [ w, z]) + g_{t}(\tau)( w, [ z, u]) + g_{t}(\tau)( z, [ u, w]).
\end{align}
Since $\nabla_{t}(\tau)$ and $\nabla_{t}(t) = \nabla (t)$ are both defined on $M_{t}$ it makes sense to consider the time rate
\begin{align}\label{var_conn_definition}
( \delta \nabla )(u,w) &=  \frac{\partial \nabla_{t}(\tau)}{\partial \tau}|_{\tau = t}(u,w) ={\rm lim}_{\tau \rightarrow t} \frac{1}{\tau - t} \left\{ \nabla_{t}(\tau)(u,w) - \nabla_{t}(t)(u,w) \right\}. 
\end{align}
Since the difference of the two connections $\nabla_{t}(\tau) - \nabla_{t}(t)$ at any point is a  tensor field of type (1,2), the variation $\delta \nabla $ is a tensor of the same order. \\ 
The following variation formulas for the metric tensor field and the unit normal vector field have been proved in \cite{NKad_Ftravlo}:
\begin{align}
\delta g &= -2v_{n}B + {\pounds}_{v^{||}}g=2 {\mathcal D}^{\flat} \label{var_metric_1},\, \, \, \delta n = {\ W}n = -JSv^{||} -J \nabla v_{n}%\label{var_metric_D}
\end{align}
\section{Variation of the Levi - Civita connection of a moving hypersurface.}
In this section we present the main results of this paper. We need the following lemma 
\begin{lemma}\label{lemma_5_1_aux}
Consider the extentions $\overline{u}, \overline{w}$ of $u$, $w$ which are defined by the motion with velocity $v$ and satisfying (\ref{extend_vf_by_motion_defn}),(\ref{extend_vf_by_motion_eqn1a}) and (\ref{extend_vf_by_motion_eqn2a}). Then: 
\begin{align}\label{sec2_lemma_eqn1c} 
\frac{\partial}{\partial \tau}|_{\tau =t} \overline{\nabla}_{F_{t}(\tau)u}F_{t}(\tau)w &= \overline{\nabla}_{v} \overline{\nabla}_{\overline{u}}\overline{w} = \overline{R}(v, Ju)\overline{w} + \overline{\nabla}_{Ju}Gw.
\end{align} 
\end{lemma}
\begin{proof}
Since time derivatives along the motion  are covariant derivatives in the direction of the velocity vector field we have that:
\begin{align*}
\frac{\partial}{\partial \tau}|_{\tau =t} \overline{\nabla}_{F_{t}(\tau)u}F_{t}(\tau)w &= \overline{\nabla}_{v}\overline{\nabla}_{\bar u}\overline{w} 
= \overline{R}(v, \overline{u})\overline{w} + \overline{\nabla}_{\overline{u}}\overline{\nabla}_{v}\overline{w} + \overline{\nabla}_{[v, \overline{u}]}\overline{w} \\
&= \overline{R}(v, \overline{u})\overline{w} + \overline{\nabla}_{\overline{u}}\overline{\nabla}_{v}\overline{w},
\end{align*}
Using the relations (\ref{extend_vf_by_motion_defn}),(\ref{extend_vf_by_motion_eqn1a}), (\ref{extend_vf_by_motion_eqn2a}) and (\ref{G defined})
%\begin{align*}
%\overline{u}_{x} &= Ju_{x}, \, \, \overline{w}_{x} = Jw_{x}, \, \,  [v, \overline{u}] = [v, \overline{w}] =0
%\end{align*}
as well as the relations
\begin{align*}
\overline{\nabla}_{v} \overline{w} &= \overline{\nabla}_{\overline{w}}v +[v, \overline{w}]= \overline{\nabla}_{Jw}v =Gw
\end{align*}
and
\begin{align*}
\overline{R}(v, \overline{u})\overline{w} &= \overline{R}(v, Ju)\overline{w},
\end{align*}
formula (\ref{sec2_lemma_eqn1c}) follows.
\end{proof}
\begin{proposition}
For any $u,w \in{\mathcal X}(M)$ and their extensions $\bar{u},\bar{w} \in {\mathcal X}(N)$, the variation of the Levi - Civita connection of a hypersurface $M$ moving in a Riemannian manifold $N$ with velocity $v$ is given by the following equivalent expressions: 
\begin{align}\label{sec2_var_con_eqn2a}
( \delta \nabla ) (u,w) &= - {\mathcal P}G \nabla_{u}w + {\mathcal P}\overline{\nabla}_{Ju}Gw - B(u,w){\mathcal P}Wn + {\mathcal P}\overline{R}(v, Ju)\overline{w},
\end{align}
\begin{align}\label{sec2_var_Christof_eqn1d}
( \delta \nabla ({u},w))^{\flat}(z) &= -2 {\mathcal D}^{\flat}( \nabla_{u}w,z ) + u ( {\mathcal D}^{\flat}(w,z))+ w({\mathcal D}^{\flat}(u,z))  - z ({\mathcal D}^{\flat}(u,w)) \nonumber \\
&- {\mathcal D}^{\flat}(u, [w,z]) + {\mathcal D}^{\flat}(w, [z,u])+ {\mathcal D}^{\flat}(z, [u,w]). \\
\label{sec2_var_Christof_eqn1d_A}
g (\delta \nabla (u,w),z) &= g( (\nabla_{u}{\mathcal D})w,z)+ g((\nabla_{w}{\mathcal D})u, z) - g((\nabla_{z}{\mathcal D})u, w). 
\end{align}
\end{proposition}
\begin{proof}
The first formula is proved using the definition (\ref{sec2_var_con_eqn1b}) and lemma (\ref{lemma_5_1_aux}). We have:
\begin{align*}
\frac{\partial}{\partial \tau}|_{\tau =t} F_{t}(\tau) {\nabla}_{F_{t}(\tau)u}F_{t}(\tau)w &= \frac{\partial}{\partial \tau}|_{\tau =t} \overline{\nabla}_{F_{t}(\tau)u}F_{t}(\tau)w \\
&- \frac{\partial}{\partial \tau}|_{\tau =t} \left\{ \overline{g} \left( \overline{\nabla}_{F_{t}(\tau)u}F_{t}(\tau)w, n(\tau) \right)n(\tau) \right\}
\end{align*} 
hence
\begin{align}\label{sec2_var_con_eqn2c}
G \nabla_{u}w + J (\delta \nabla)(u,w)&= \overline{R}(v, \overline{u})\overline{w} +  \overline{\nabla}_{\overline{u}}Gw \nonumber \\
&- \left\{ \overline{g}\left( \frac{\partial}{\partial \tau}|_{\tau =t}\overline{\nabla}_{F_{t}(\tau)u}F_{t}(\tau)w, n(t) \right) + \overline{g}\left( \overline{\nabla}_{F_{t}(\tau)u}F_{t}(\tau)w, \delta n \right) \right\} n(t)  \nonumber \\
&-  \overline{g} ( \overline{\nabla}_{\overline{u}} \overline{w},n(\tau) )  \delta n.  
\end{align}
Applying the projection ${\mathcal P}_{t}$ on $M_{t}$ on the last relation we get
%\begin{align*}
%( \delta \nabla ) (u,w)&= - {\mathcal P}G \nabla_{u}w + {\mathcal P}\overline{\nabla}_{Ju}Gw - B(u,w){\mathcal P}Wn + {\mathcal P}\overline{R}(v, Ju)Jw,
%\end{align*} 
(\ref{sec2_var_con_eqn2a}). \\
Formula (\ref{sec2_var_Christof_eqn1d}) is an immediate consequence of  (\ref{sec2_var_Christof_eqn1b})and (\ref{var_metric_1}).
%of the relation ($\delta g = 2{\mathcal D}^{\flat}$) giving the variation of the metric tensor field of $M$ in terms of the kinematical tensor field ${\mathcal D}$ .\\
To prove (\ref{sec2_var_Christof_eqn1d_A}) we observe that the terms involved in equation (\ref{sec2_var_Christof_eqn1d}) can be further decomposed as follows:
\begin{align*}
-2 g ({\mathcal D}\nabla_{u}w, z) &= - g ({\mathcal D}\nabla_{u}w, z)- g ({\mathcal D}\nabla_{u}w, z), \, {\bf (i)}\\ 
u (g({\mathcal D}w, z) ) &= g((\nabla_{u}{\mathcal D})w,z)+g({\mathcal D}\nabla_{u}w, z)+ g({\mathcal D}w, \nabla_{u}z), \, {\bf (ii)} \\ 
w(g({\mathcal D}u, z)) &= g( (\nabla_{w}{\mathcal D})u,z) + g({\mathcal D}\nabla_{w}u,z)+ g({\mathcal D}u, \nabla_{w}z) \, {\bf (iii)}\\ 
-zg({\mathcal D}u,w) &= - g((\nabla_{z}{\mathcal D})u,w)- g({\mathcal D}\nabla_{z}u,w)- g({\mathcal D}u, \nabla_{z}w) \, {\bf (iv)}\\
- g({\mathcal D}u, [w,z]) &= - g({\mathcal D}u, \nabla_{w}z) + g({\mathcal D}u, \nabla_{z}w) \, {\bf (v)} \\ 
g({\mathcal D}w, [z,u]) &= g({\mathcal D}w, \nabla_{z}u) - g({\mathcal D}w, \nabla_{u}z), \, {\bf (vi)} \\ 
g({\mathcal D}z, [u,w] ) &= g({\mathcal D}z, \nabla_{u}w) - g({\mathcal D}z, \nabla_{w}u), \, {\bf (vii)}
\end{align*}
Using the $g$ - symmetry of the rate of deformation ${\mathcal D} $ and adding the terms of the relations ${\bf (i)}$ -${\bf (vii)}$ we derive (\ref{sec2_var_Christof_eqn1d_A}).
%\begin{align*}
%g( (\delta \nabla)(u,w), z ) &=  g( (\nabla_{u}{\mathcal D})w,z)+ g((\nabla_{w}{\mathcal D})u, z) - g((\nabla_{z}{\mathcal D})u, w).
%\end{align*}
\end{proof}
The next formula gives the variation of the connection in terms of geometrical quantities.
\begin{proposition}
For any $u,w\in{\mathcal X}(M)$ and their extensions $\bar{u}, \bar{w} \in{\mathcal X}(N)$, the variation of the Levi - Civita connection of a hypersurface moving with velocity field $v = Jv^{||} + v_{n}n$, is given by the formula:
\begin{align}\label{sec2_var_generic_eqn2b}
( \delta \nabla )(u,w) &= - v_{n} (\nabla_{u}S)w - \left\{ w(v_{n})Su + u(v_{n})Sw \right\} + B(u,w) \nabla v_{n}  \nonumber \\
&+ v_{n} {\mathcal P}\overline{R}(n, Ju)\overline{w} + {( \pounds}_{v^{||}}\nabla )_{u}w .
\end{align}
\end{proposition}
\begin{proof}
We start by decomposing each  term of (\ref{sec2_var_con_eqn2a}) using (\ref{Gauss_equation_hypersurface}), (\ref{Gauss_equation_hypersurface_n}) and (\ref{G=JD+WJ}):%( since ${\mathcal P}G = \nabla v^{||} -v_{n}S$ and $\delta n = Wn = - \overline{\nabla}v_{n} - JSv^{||}$, 
\begin{align}\label{sec2_var_con_eqn2d}
-{\mathcal P}G \nabla_{u}w &= v_{n}S \nabla_{u}w - \nabla_{\nabla_{u}w} v^{||}= v_{n}S \nabla_{u}w - \nabla_{v^{||}}\nabla_{u}w - [\nabla_{u}w, v^{||}] \nonumber \\
%&= v_{n}S \nabla_{u}w - \nabla_{v^{||}}\nabla_{u}w - [\nabla_{u}w, v^{||}] - R(v^{||},u)w \nonumber \\
%&- \nabla_{u}\nabla_{v^{||}}w - \nabla_{[v^{||},u]}w - [\nabla_{u}w, v^{||}] \nonumber \\
&=v_{n}S \nabla_{u}w + [v^{||}, \nabla_{u}w] - \nabla_{[v^{||},u]}w - \nabla_{u}\nabla_{v^{||}}w - R(v^{||},u)w,
\end{align}
\begin{align}   
\label{sec2_var_con_eqn2e}
{\mathcal P}\overline{\nabla}_{Ju}Gw &= \nabla_{u}{\mathcal P}Gw - (B(v^{||},w)+w(v_{n}))Su \nonumber \\
&= - v_{n}\nabla_{u}Sw - u(v_{n})Sw - w(v_{n})Su + \nabla_{u}\nabla_{w}v^{||} - g(Sv^{||},w)Su, 
 \end{align}
 \begin{align}
\label{sec2_var_con_eqn2f}
- B(u,w) {\mathcal P}Wn &= B(u,w) \nabla v_{n} + g(Su,w)Sv^{||} ,
\end{align}
\begin{align}
\label{sec2_var_con_eqn2g}
{\mathcal P} \overline{R}(v, Ju)\overline{w} &= v_{n}{\mathcal P}\overline{R}(n, Ju)\overline{w} + {\mathcal P}\overline{R}(Jv^{||}, Ju)\overline{w} \nonumber \\
&=v_{n}{\mathcal P} \overline{R}(n, Ju)\overline{w} + R(v^{||},u)w + g(Sv^{||},w)Su - g(Su, w)Sv^{||} .
\end{align} 
therefore (\ref{sec2_var_con_eqn2a}) becomes
\begin{align}\label{sec2_var_con_eqn2g_AA}
( \delta \nabla )(u,w) &= v_{n}S \nabla_{u}w + [v^{||}, \nabla_{u}w] - \nabla_{[v^{||},u]}w - \nabla_{u}\nabla_{v^{||}}w - R(v^{||},u)w \nonumber \\
&- v_{n}\nabla_{u}Sw - u(v_{n})Sw - w(v_{n})Su + \nabla_{u}\nabla_{w}v^{||} - g(Sv^{||},w)Su \nonumber \\
&+ B(u,w) \nabla v_{n} + g(Su,w)Sv^{||} + v_{n}{\mathcal P} \overline{R}(n, Ju)\overline{w} + R(v^{||},u)w \nonumber \\
&+ g(Sv^{||},w)Su - g(Su, w)Sv^{||}
\nonumber \\ 
&= -v_{n}(\nabla_{u}S)w - \{ u(v_{n})Sw + w(v_{n})Su \} + B(u,w) \nabla v_{n} \nonumber \\
&+ v_{n}{\mathcal P} \overline{R}(n, Ju)\overline{w} + [v^{||}, \nabla_{u}w] - \nabla_{[v^{||},u]}w - \nabla_{u}\nabla_{v^{||}}w \nonumber \\
&- R(v^{||},u)w + \nabla_{u}\nabla_{w}v^{||} + R(v^{||},u)w 
\end{align}
%where
%\begin{align}\label{sec2_var_con_eqn2g_AB}
%- \nabla_{\nabla_{u}w} v^{||} &= 
%\end{align}
which reduces to:
\begin{align}\label{sec2_var_con_eqn2g_AC}
( \delta \nabla )(u,w) &= %- v_{n} (\nabla_{u}S)w - \left\{ w(v_{n})Su + u(v_{n})Sw \right\} + B(u,w) \nabla v_{n}  \nonumber \\
%&+ v_{n} {\mathcal P}\overline{R}(n, Ju)Jw + [v^{||}, \nabla_{u}w] - \nabla_{[v^{||},u]}w - \{ \nabla_{u}\nabla_{v^{||}}w - \nabla_{u}\nabla_{w}v^{||}\} \nonumber \\
%&= 
- v_{n} (\nabla_{u}S)w - \left\{ w(v_{n})Su + u(v_{n})Sw \right\} + B(u,w) \nabla v_{n}  \nonumber \nonumber \\
&+ v_{n} {\mathcal P}\overline{R}(n, Ju)\overline{w} + [v^{||}, \nabla_{u}w] - \nabla_{[v^{||},u]}w - \nabla_{u}[v^{||},w].
\end{align}
Since the Lie derivative of the connection $\nabla$ is given by (\cite{Yano})
\begin{align}\label{Lie_der_connection}
({\pounds}_{v^{||}}\nabla )(u,w) &= [v^{||}, \nabla_{u}w] - \nabla_{[v^{||},u]}w - \nabla_{u}[v^{||},w],
\end{align}  
%then, using (\ref{Lie_der_connection}) and (\ref{sec2_var_con_eqn2g_AC}) 
we finally get (\ref{sec2_var_generic_eqn2b}).
\end{proof}
Since covariant derivative commutes with the index lowering operation, we get the following:
\begin{corollary}
Formula (\ref{sec2_var_Christof_eqn1d_A}) can be written in the equivalent form:
\begin{align}\label{sec2_var_Christof_eqn1d_A_altern1a}
g( (\delta \nabla) ({u},w), z ) &= \left(  \nabla _{u}{\mathcal D}^{\flat}\right) (w,z)+ \left(  \nabla _{w}{\mathcal D}^{\flat}\right) (u,z) -  \left(  \nabla _{z}{\mathcal D}^{\flat}\right) (u,w)
\end{align}
\end{corollary}
Similarly to the concept of an infinitesimal isometry ($\delta g= 0$) we define the following concept:
\begin{definition}\label{defn_affine_motion}
We say that a motion $\phi_{t}(\tau)$ of a hypersurface $M$ is infinitesimally affine, if $\delta \nabla = 0$.
\end{definition}
We now give a necessary and sufficient condition in order that a motion is infinitesimally affine.  
\begin{proposition}\label{proposition_char_inf_aff_D_const}
A motion is infinitesimally affine if and only if $\nabla_{u}{\mathcal D}=0$, for any $u\in {\mathcal X}(M_{t})$.
\end{proposition}  
\begin{proof}
Let $\nabla_{u} {\mathcal D} =0$ for any $u \in {\mathcal X}(M_{t})$, then the righthand side of (\ref{sec2_var_Christof_eqn1d_A}) is zero therefore $\delta \nabla =0$. \\
Let us now assume that $\delta \nabla =0$, then the right hand side of (\ref{sec2_var_Christof_eqn1d_A}) is zero for any $u, w, z \in {\mathcal X}(M_{t})$, hence:
\begin{align*}
g( (\nabla_{u}{\mathcal D})w, z) +  g( (\nabla_{w}{\mathcal D})u, z) -  g( (\nabla_{z}{\mathcal D})u, w) &= 0 \, \, {\bf (i)}
\end{align*}
Permutating cyclically $u, w, z \in {\mathcal X}(M_{t})$ we get the two relations:
\begin{align*}
g( ( \delta \nabla )(w,z), u ) &= g( (\nabla_{w}{\mathcal D})z, u) +  g( (\nabla_{z}{\mathcal D})w, u) -  g( (\nabla_{u}{\mathcal D})w, z) \, \, {\bf (ii)} \\
g( ( \delta \nabla )(z,u), w ) &= g( (\nabla_{z}{\mathcal D})u, w) +  g( (\nabla_{u}{\mathcal D})z, w) -  g( (\nabla_{w}{\mathcal D})u, z) \,\, {\bf (iii)}
\end{align*}
Using the $g$ - symmetry of $\nabla_{u} {\mathcal D}$, $\forall u \in {\mathcal X}(M_{t})$ and subtracting $({\bf iii})$ from the sum of ${\bf (i)}$ and ${\bf (iii)}$, we get:
\begin{align*}
2g ( (\nabla_{u}{\mathcal D})w, z ) - \left\{g((\nabla_{w}{\mathcal D})z,u) + g((\nabla_{z}{\mathcal D})w,u) - g((\nabla_{u}{\mathcal D})z,w) \right\} = 0.   
\end{align*}
But from $({\bf ii})$ the expression in the bracket is zero, therefore
\begin{align*}
g ( (\nabla_{u}{\mathcal D})w, z ) &= 0, 
\end{align*}  
for any $u, w, z \in {\mathcal X}(M_{t})$.
\end{proof}
Hypersurfaces admitting a parallel second order tensor field have been studied  before (\cite{Eisenhart}, \cite{Vilms}). 
Further, from formula (\ref{var_metric_1}) and the fact that  the covariant derivative commutes with the index lowering operation we get:
\begin{corollary}
A motion is infinitesimally affine if and only if the stretching ${\mathcal D}^{\flat}$, equivalently the variation of the metric $\delta g$, is parallel. 
\end{corollary}
%\begin{remark}
%Clearly, if the ambient manifold is an Euclidean space, then any affine %motion is also infinitesimally affine.
%\end{remark}
Next we examine some special motions.

If the motion is tangential ($v = J v^{||}$) then $\delta \nabla =  {\pounds}_{v^{||}}\nabla$, hence we have the condition for the affine Killing fields (\cite{Yano}, p. 24):
\begin{corollary}
A tangential motion of $M$ is infinitesimally affine if and only if its velocity field is an affine Killing field.
\end{corollary}
For a normal motion ($v=v_{n}n$) formula (\ref{DandS}) implies that ${\mathcal D} = -v_{n}S$, i.e.  ${\mathcal D}^{\flat} = -v_{n}B$. Therefore the formulas (\ref{sec2_var_Christof_eqn1d_A}) and (\ref{sec2_var_generic_eqn2b}) reduce to the following:
\begin{corollary}
For a normal motion the variation of the connection is given by:
\begin{align} 
\label{sec2_var_Christof_eqn1d_A_normal}
 g( (\delta \nabla)({u},w), z ) &= -\left(  \nabla _{u}{v_{n}B}\right) (w,z)- \left(  \nabla _{w}{v_{n}B}\right) (u,z) + \left(  \nabla_{z}{v_{n}B}\right) (u,w) \\ 
g( (\delta \nabla)({u},w), z ) &= -v_{n} \left\{ (\nabla_{u}B)(w,z)+ (\nabla_{w}B)(u,z) - (\nabla_{z}B)(u,w) \right\} \nonumber \\
&+ \left\{ u(v_{n})B(w,z) + w(v_{n})B(u,z) - z(v_{n})B(u,w)\right\}. \label{sec2_var_Christof_eqn1d_A_normal_BB}
\end{align}
\end{corollary}
An immediate consequence of (\ref{sec2_var_Christof_eqn1d_A_normal}) is the following known relation between the velocity and the extrinsic geometry of the hypersurface  (theorem 4 in \cite{Yano&Chen} and corollary 2 in \cite{Yano2}).
\begin{corollary}
A necessary and sufficient condition for a normal motion to be infinitesimally affine is $ \nabla v_{n}B =0$.	
\end{corollary}
If the ambient space is Euclidean ($N = \mathbb R^{m+1}$) then, using the equation of Codazzi, equation (\ref{sec2_var_Christof_eqn1d_A_normal_BB}) reduces to 
\begin{align}\label{sec2_var_Christof_eqn1d_A_normal_B_Eucl_amb}
g( (\delta \nabla) ({u},w), z ) &= -v_{n} (\nabla_{u}B)(w,z) \nonumber \\
&+ \left\{ u(v_{n})B(w,z) + w(v_{n})B(u,z) - z(v_{n})B(u,w)\right\}.
\end{align}
\begin{remark}
Parallel hypersurfaces in Euclidean space are produced by a special case of a normal motion in which $v_{n}$ does not depend on the position ($\nabla v_{n}=0$). In this case the hypersurface is moving along its normal the same distance at any point and parallel to itself. Then formula (\ref{sec2_var_Christof_eqn1d_A_normal_B_Eucl_amb}) reduces further to:
\begin{align}\label{extend_vf_by_motion_eqn2f}
( \delta \nabla )({u},w) &=  -v_{n}(\nabla_{u}S)w.
\end{align}
\end{remark}
An immediate consequence of (\ref{extend_vf_by_motion_eqn2f}) is the following: 
\begin{corollary}
In the parallel motion of a hypersurface in a Euclidean space, the affine connection is infinitesimally preserved if and only if the surface has parallel shape operator.
\end{corollary}
\begin{remark}
The above corollary imposes a very strong restriction for the hypersurface in order to admit a motion parallel to itself and infinitesimally affine \cite{Vilms}.
\end{remark}
\section{Some examples}
In this section we show how the calculations of this work apply to specific cases. 
\begin{example}{Parallel hypersurfaces in Euclidean space:} 
In the case of a hypersurface in an Euclidean space moving parallel to itself, we can have explicit forms of its kinematical and geometrical quantities. This allows us to directly calculate the variation of the connection either using its definition (\ref{sec2_var_con_eqn1b}) or the compatibility condition (\ref{sec2_var_Christof_eqn1b}). \\
Let $M^{}$ be an oriented hypersurface in $\mathbf R^{m+1}$ with unit normal field $n$. Suppose that $M$ undergoes a {\it parallel motion} given by the mapping  $\phi_{t}(\tau):M_{t} \rightarrow \mathbf R^{m+1}$ such that:
\begin{align}\label{parallel_motion_defn}
\phi_{t}(\tau)(p) &= j_{t}(p) + \epsilon_{t}(\tau)n(p), \, \,   \tau \geq{t}
\end{align} 
where $\epsilon_{t}(\tau)$ is a differentiable real valued function of $t$ and $\tau$ with $\epsilon_{t}(t) =0$. 
The velocity of the motion is $v = v_{n}n$ where $v_{n} = \frac{\partial \epsilon_{t}(\tau)}{\partial \tau}|_{\tau =t}$. The deformation gradient $F_{t}(\tau)(p)$ (differential of $\phi_{t}(\tau)(p)$), is given by:
\begin{equation}\label{diff_parall_motion}
F_{t}(\tau)=J ( I - \epsilon_{t}(\tau)S )
\end{equation}
All the kinematical quantities are derived from expression (\ref{diff_parall_motion}) (see: \cite{NKadianakis}). Using (\ref{D(t)}), (\ref{extend_vf_by_motion_eqn2b}), (\ref{G=JD+WJ}), (\ref{DandS}) and (\ref{gtt}), we get the metric tensor, the velocity gradient, and the stretching of the motion:
 \begin{align}\label{t_dep_metric_par_motion}
g_{t}(\tau)&= g(t) - 2\epsilon_{t}(\tau)B + \epsilon_{t}(\tau)^{2}III \\
G(t) &= -v_{n}(t)JS(t) \label{t_dep_G_par_motion}\\
{\mathcal D}(t) &= -v_{n}(t)S(t), \label{stretching_parallel} 
\end{align}            
where $III$ is the third fundamental form of the hypersurface.
The motion induces a $\tau$ - dependent connection $\nabla_{t}(\tau)$ on $M_{t}$ given by (\ref{sec2_var_con_eqn1b}) and satisfying the compatibility condition (\ref{sec2_var_Christof_eqn1b}). Using this condition, the above expression for the $\tau$ -depending metric, the Codazzi equation and the $g$ - symmetry of $\nabla_{X} S$, we get for any vector fields $X,Y,Z \in {\mathcal X}(M)$:
\begin{align}\label{paral_mot_conn_var_eqn1a}
&2g(\nabla_{t}(\tau)_{Y}X,Z) -4 \epsilon_{t}(\tau)g(S(t) \nabla_{t}(\tau)_{Y}X, Z) + \epsilon^{2}_{t}(\tau)g(S^{2}(t)\nabla_{t}(\tau)_{Y}X,Z)  \nonumber \\
&= 2 g ( \nabla_{Y}X,Z ) - 4 \epsilon_{t}(\tau) g \left( (\nabla_{Y}S)X,Z \right) - 4 \epsilon_{t}(\tau)g(S(t)\nabla_{Y}X, Z) \nonumber \\
&+ \epsilon^{2}_{t}(\tau){\mathcal Q}_{}(X,Y,Z)
\end{align}
where the quantity ${\mathcal Q}_{}(X,Y,Z)$  is given by:
\begin{align*}
{\mathcal Q}_{}(X,Y,Z) &= g( (\nabla_{X}S^{2})Y,Z ) + g((\nabla_{Y}S^{2})X,Z) - g( (\nabla_{Z}S^{2})X,Y ) \nonumber \\
&+ g(2 S^{2}\nabla_{Y}X, Z ). 
\end{align*}
Taking the derivative on both sides of (\ref{paral_mot_conn_var_eqn1a}) with respect to $\tau$ at $\tau=t$ we get:
\begin{align}\label{paral_mot_conn_var_eqn1b}
2 g ( ( \delta \nabla ) (Y,X),Z ) &= - 2v_{n} g \left( (\nabla_{Y}S)X,Z \right).
\end{align}
Since the relation (\ref{paral_mot_conn_var_eqn1b}) is valid for any choice of $X,Y,Z \in {\mathcal X}(M)$ we get again formula (\ref{extend_vf_by_motion_eqn2f}). 
\end{example}
\begin{example}{{Balloon expansion:}}
We consider a spherical balloon expanding by blowing air in at the point O. We assume that this is modeled by a homothetic motion of the sphere $S_{1} : x^{2} + y^{2}+ (z-1)^{2} =1$ with radius $1$ and center at $(0,0,1)$. We study its deformation under the affine map of the ambient space $(x,y,z) \rightarrow t (x,y,z)$.
\begin{figure}[htbp]
\centering
\includegraphics[width=0.50 \textwidth,bb=0 0 273 278,clip=true]{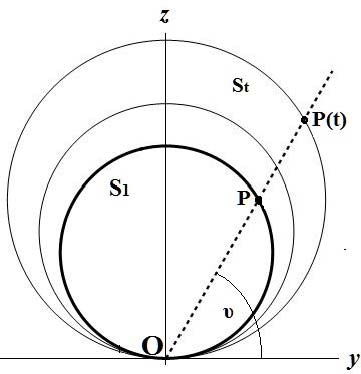}
\caption{Balloon expansion}
\label{Balloon expansion}
\end{figure}
This creates the homothetic family of spheres $S_{t}: x^{2} + y^{2} + (z- t)^{2} = t^{2}$ (see Figure 1 with one dimension suppressed).
Each of these spheres $S_{t}$ is a result from a rotation around the $z$ - axis of the circle 
$ C_{t}: x = 0, \, \, y^{2}+(z-t)^{2} =t^{2}$.  
Using the polar angle $v$ as a parameter we have the parametrization for this circle 
$ c(v) = (0,t {\rm sin }2v , 2t {\rm sin}^{2}v) $.
Thus we have the parametrization, of the sphere $S_{t}$ at time $t$, or equivalently the motion: 
\begin{align}\label{Example1_eqn1b}
\phi: S_{1} \rightarrow \mathbf R^{3}, \, \, \phi (u,v, t) = t \left({\rm cos }u \, {\rm sin}2v, {\rm sin }u \, {\rm sin}2v, 2{\rm sin}^{2}v \right), \, \, t \geq 1.
\end{align}
This motion maps the point $P$ at $t=1$ to the point $P(t)$ (Figure 1).\\
Clearly $j(u,v)=\phi(u,v,1) = \left({\rm cos }u \, {\rm sin}2v, {\rm sin }u \, {\rm sin}2v, 2{\rm sin}^{2}v \right)$ is the canonical injection of the original sphere and $J=dj$ its differential. We consider the base on $T_{p}S_{1}$ consisting of the partial derivatives of $j$ i.e. $\{ \epsilon_{1} = j_{u}, \, \epsilon_{2} = j_{v}\}$  and the usual orthonormal basis $\{e_{1}, e_{2}, e_{3}\}$ in $\mathbf R^{3}$. Then, relative to these bases we have that the matrices of the linear maps $J(u,v) =dj(u,v)$ and $F(u,v,t)= d \phi (u,v,t)$ are:      
\begin{align}\label{Example1_eqn1d}
[J(u,v)] &= 
\left[
\begin{array}{cc}
-{\rm sin}u \, {\rm sin} 2v & 2 {\rm cos} u \, {\rm cos} 2v \\
{\rm cos}u \, {\rm sin} 2v & 2 {\rm sin} u \, {\rm cos} 2v \\
0 & 2 {\rm sin} 2v
\end{array}
\right], \nonumber \\
 \\ [F(u,v)] &= t\left[
\begin{array}{cc}
-{\rm sin}u \, {\rm sin} 2v & 2 {\rm cos} u \, {\rm cos} 2v \\
{\rm cos}u \, {\rm sin} 2v & 2 {\rm sin} u \, {\rm cos} 2v \\
0 & 2 {\rm sin} 2v
\end{array}
\right] = t [J(u,v)]. \nonumber  
\end{align}
The metric $g$ on the original sphere and the  $t$ - dependent metric on $S_{1}$ have matrices given by: 
\begin{align}\label{Example1_eqn1f_0}
[ g(u,v) ] &=  
\left[ 
\begin{array}{cc}
{\rm sin}^{2}2v & 0 \\
0 & 4
\end{array}
\right], \, \, \, [ g_{t}(u,v) ] = t^{2}[g (u,v)] 
\end{align}
respectively. \\
Using relation (\ref{var_metric_1}) one directly finds the stretching 
\begin{align}\label{Example1_eqn3a}
[{\mathcal D}^{\flat}(u,v)] &= [\frac{1}{2}\delta g] = [\frac{1}{2} \frac{\partial g(t)}{\partial t}|_{t=1}] = [g(u,v)].
\end{align}
i.e. ${\mathcal D}^{\flat}$ coincides with the metric and thus it is parallel. Since $(\nabla_{u}{\mathcal D}^{\flat} ) = (\nabla_{u}{\mathcal D}^{} )^{\flat}$ it follows that ${\mathcal D}$ is parallel. \\
Further, the non zero Christofell symbols of the connection, relative to the bases we considered above, are: 
\begin{align*}
\Gamma_{12}^{1} &= \Gamma_{21}^{1} = 2 {\rm cot} 2v, \, \, \Gamma_{11}^{2} = - \frac{1}{4}{\rm sin} 4v, 
\end{align*} 
Since they do not depend on $t$, we have $\delta \nabla =0$. \\
The stretching ${\mathcal D}$ may also be calculated from the deformation $F(u,v,t)$ using formulas (\ref{t=tau1}) - (\ref{D(t)}). Indeed, the adjoint of $F(u,v,t)$ can be found using the metrics $g(u,v)$ on $M$ and $\bar{g}$ of $\mathbf R^{3}$. It has a matrix which is the product of the matrix of the inverse of the metric $g(u,v)$ with the transpose of the matrix of $F(u,v,t)$. Using this we find:
\begin{align*}
[F^{*}_{t}(u,v,t)] &= [g^{-1}][F(u,v,t)]^{T} \\
&=  \left[ 
\begin{array}{cc}
\frac{1}{{\rm sin}^{2}2v} & 0 \\
0 & \frac{1}{4}
\end{array}
\right] t \left[
\begin{array}{cc}
-{\rm sin}u \, {\rm sin} 2v & 2 {\rm cos} u \, {\rm cos} 2v \\
{\rm cos}u \, {\rm sin} 2v & 2 {\rm sin} u \, {\rm cos} 2v \\
0 & 2 {\rm sin} 2v
\end{array}
\right]^{T} \\
&= t \left[ 
\begin{array}{ccc}
\frac{-sin u}{sin 2 v} & \frac{cos u}{sin 2v} & 0 \\ 
& &  \\
\frac{cos u \, cos 2v}{2} & \frac{sinu \, cos 2v}{2} & \frac{sin \, 2v}{2}
\end{array}
\right]
\end{align*}
and thus the Cauchy - Green tensor field has the matrix 
\begin{align}\label{Example1_eqn2f}
[{ C}_{t}(u,v)] &= [F^{*}_{t}(u,v)F_{t}(u,v) ] = t^{2}  
\left[ 
\begin{array}{cc}
1 & 0 \\
0 & 1
\end{array}
\right]. 
\end{align}
Then differentiating with respect to $t$ at $t=1$ we get:
\begin{align}\label{D_hom_spheres_altern_eq1}
[{\mathcal D}] &= \frac{1}{2}\left[ \frac{dC}{dt}|_{t=1}\right] = \left[ 
\begin{array}{cc}
1 & 0 \\
0 & 1
\end{array}
\right].
\end{align}
We observe that although $[{\mathcal D}]$ has constant components and $[{\mathcal D}^{\flat}]$ has not, both are covariantly constant since $[{\mathcal D}^{\flat}] = [g]{[\mathcal D}]$. Therefore the motion is infinitesimally affine but not infinitesimally isometric.
\end{example}
\begin{example}{An infinitesimally affine unrolling of a cylindrical shell}. 
Let $K_{t}$ be the circular cylinder defined by its axis (the x-axis) and the circle $ C_{t}: x = 0, \, \, y^{2}+(z-t)^{2} =t^{2}$. By an {\it unrolling} of the cylinder $K_{1}$ we mean (see figure 2, with x-axis suppressed) that a point $P$ on any circular section of the cylinder $K_{1}$ with a plane parallel to the $yz$ plane, is mapped to the point $P^{}(t)$ on the circular section by the same plane, of the cylinder $K_{t}$, such that the arcs $OP$ and $OP(t)^{}$ have the same length. Using the polar angle $u$ as a parameter, we have, as in the previous example, the parametrization of  $ C_{t}$: $c(u)=(0,t{\rm sin} \,2{u} , \, t{\rm sin}^{2}u)$. 
The circular cylinder $K_{t}$ has the parametrization $k(u,v,t)=(v,t{\rm sin} \,2{u} , \, t{\rm sin}^{2}u)$.\\
The polar angles $u$ of $P$ and $u(t)$ of $P^{}(t)$ are then related by $tu(t) = u$. Thus we have the following isometric unrolling of the cylinder: 
\begin{align}\label{unfolding_right_cylinder_eqn1b}
f(u,v,t) &= \left( v, t \, {\rm sin} \,\frac{2u}{t} , \, 2t{\rm sin}^{2}\, \frac{u}{t} \right),\,\,\, t \geq 1.
\end{align}
\begin{figure}[htbp]
\centering
\includegraphics[scale=0.5,bb=0 0 404 311,clip=true]{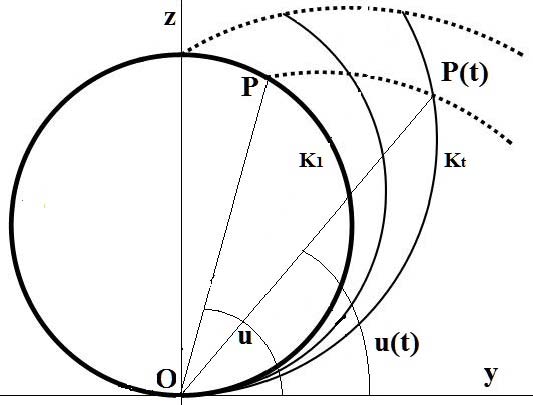}
\caption{Unrolling of a piece of a cylindrical shell}
\label{elongation_unfolding_cylinder}
\end{figure}
In order to obtain an infinitesimally affine unrolling we combine the previous unrolling with a time dependent  elongation along the axis of the cylinder $K_{t}$ so we finally have the motion:
\begin{align}\label{deform_right_cylinder_eqn1b}
\phi(u,v,t) &= \left( t v, t \, {\rm sin} \,\frac{2u}{t} , \, 2t{\rm sin}^{2}\, \frac{u}{t} \right),\,\,\, t \geq 1.
\end{align}
The canonical injection of the original cylinder is $j(u,v)=\phi(u,v,1)$. Considering the partial derivatives of $j$ as the base vectors on the tangent space of the cylinder and the usual basis on the ambient space we have the following matrices for the original metric on the surface, the time dependent metric and the $(0,2)$ tensor field ${\mathcal D}^{\flat}$: 
\begin{align*}
[g(u,v,1)] &= \left[ \begin{array}{cc} 4 & 0 \\ 0 & 1 \end{array} \right], \, \, [g(u,v,t)] = \left[ \begin{array}{cc} 4 & 0 \\ 0 & t^{2} \end{array} \right], \\
[D^{\flat}(u,v)] &= \frac{1}{2} (\delta g)(u,v) =\left[ \begin{array}{cc} 0 & 0 \\ 0 & 1 \end{array}  \right]. 
\end{align*}
The tensor field ${\mathcal D}(u,v)$ can be calculated, as in the previous example, by using the adjoint of $F(u,v,t)$ whose matrix is 
\begin{align*}
[F^{\star}(u,v,t)] &= 
\left[ \begin{array}{ccc} 
0 & \frac{1}{2}{\rm cos } \left(\frac{2u}{t}\right) & \frac{1}{2} {\rm sin} \left( \frac{2u}{t} \right) \\ t & 0 & 0
\end{array} \right]
\end{align*}
and the Cauchy - Green tensor. Both ${\mathcal D}(u,v)$ and ${\mathcal D}^{\flat}(u,v)$ are covariantly constant. On the other hand, since the time dependent metric does not depend on the position $(u,v)$ but only on the time $t$, all the Christofell symbols vanish. Therefore the motion (\ref{deform_right_cylinder_eqn1b}) is infinitesimally affine but not infinitesimally isometric.
\end{example}
%%%%%%%%%%%%%%%%%%%%%%%%%%%%%%%%%%%%%%%%%%%%%%%%%%%%%%%%%%%%%

\end{document}